\documentclass[12pt,a4paper]{amsart}

\usepackage{amsmath, nicefrac, amsthm, verbatim, amsfonts, amssymb, xcolor, bbm}
\usepackage{graphics, xspace, enumerate}
\usepackage[a4paper,margin=3cm]{geometry}
\RequirePackage[mathscr]{eucal}
\usepackage[bb=boondox]{mathalfa}

\usepackage{graphicx}
\usepackage[colorlinks=true,citecolor=green,urlcolor=green,linkcolor=green,bookmarksopen=true,unicode=true,pdffitwindow=true]{hyperref}
\usepackage[english]{babel}
\usepackage[languagenames,fixlanguage]{babelbib}
\usepackage{comment}
\hypersetup{pdfauthor={}}
\hypersetup{pdftitle={Capacity of the range of random walks on groups}}

\makeatletter
\@namedef{subjclassname@2020}{\textup{2020} Mathematics Subject Classification}
\makeatother

\theoremstyle{plain}
\newtheorem{theorem}{Theorem}[section]

\newtheorem{proposition}[theorem]{Proposition}
\theoremstyle{definition}
\newtheorem{remark}[theorem]{Remark}

\DeclareMathOperator*{\E}{\mathbf{E}}
\newcommand {\Prob} {\ensuremath{\mathbf{P}}}
\newcommand {\R} {\ensuremath{\mathbf{R}}}
\newcommand {\ZZ} {\ensuremath{\mathbf{Z}}}
\newcommand {\N} {\ensuremath{\mathbf{N}}}

\newcommand {\GG} {\ensuremath{\mathsf{G}}}

\newcommand{\Capa}{\operatorname{Cap}}
\newcommand{\Var}{\operatorname{Var}}

\newcommand{\bbN}{\mathbb{N}}

\newcommand{\bbR}{\mathbb{R}}

\newcommand{\bbjedan}{\mathbbm{1}}

\newcommand{\calC}{\mathcal{C}}

\newcommand{\calD}{\mathcal{D}}

\newcommand{\calR}{\mathcal{R}}
\newcommand{\calN}{\mathcal{N}}

\newcommand{\aps}[1]{\vert #1 \vert}

\newcommand{\floor}[1]{\lfloor #1 \rfloor}

\newcommand{\obl}[1]{\bigg( #1 \bigg)}

\newcommand{\ugl}[1]{\bigg[ #1 \bigg]}
\newcommand{\VIT}[1]{\left\{ #1 \right\}}

\numberwithin{equation}{section}

\title[Capacity of the range of random walks on groups]{Capacity of the range of random walks on groups}

\author[R.\ Mrazovi\'{c}]{Rudi Mrazovi\'{c}}
\address[Rudi\ Mrazovi\'{c}]{Department of Mathematics\\University of Zagreb\\ Zagreb\\Croatia}
\email{Rudi.Mrazovic@math.hr}

\author[N.\ Sandri\'{c}]{Nikola Sandri\'{c}}
\address[Nikola\ Sandri\'{c}]{Department of Mathematics\\University of Zagreb\\ Zagreb\\Croatia}
\email{nsandric@math.hr}

\author[S.\ \v{S}ebek]{Stjepan\ \v{S}ebek}
\address[Stjepan\ \v{S}ebek]{Department of Applied Mathematics\\
	Faculty of Electrical Engineering and Computing\\
	University of Zagreb\\ 
 Zagreb\\ 
	Croatia}
\email{stjepan.sebek@fer.hr}


\begin{document}
\allowdisplaybreaks[4]

\begin{abstract}
	In this paper, we discuss asymptotic behavior  of the capacity of the range of symmetric random walks on finitely generated groups. We show the corresponding strong law of  large numbers and central limit theorem.
\end{abstract}

\maketitle

%
%
%
%

\section{Introduction}

Let $\GG$ be a  group with  a finite symmetric set of generators $\Gamma$ (i.e.\ $\Gamma=\Gamma^{-1}$).  We define the symmetric random walk $\{S_n\}_{n\geq 0}$ on $\GG$ with respect to $\Gamma$ by taking a sequence $\{X_n\}_{n\in\N}$ of independent random elements, uniformly distributed on $\Gamma$, and defining $S_0=g$ for some fixed $g\in\GG$ and $S_{n+1}=S_nX_{n+1}$. We denote the associated probability by $\Prob_g$. Usually, we will take $g=e$ (identity) and in this case we will suppress the index and write $\Prob$. For example, taking $\GG=\ZZ^d$ and its standard set of generators $\Gamma=\{e_1,-e_1,\dots,e_d,-e_d\}$ (where $\{e_1,\dots,e_d\}$ is the standard basis of $\R^d$) gives the usual symmetric simple random walk on $\ZZ^d$. We will mention some other examples later on (we also refer the reader to \cite[Section~3.4]{peres-lyons}).

Although the distribution of the random walk $\{S_n\}_{n\geq 0}$ depends on the choice of $\Gamma$, most of its macroscopic properties are independent of the generating set one decides to work with. For example, in \cite{Varopoulos-rec-tran} it has been shown that $\{S_n\}_{n\geq 0}$ is recurrent if and only if $\GG$ contains a subgroup of finite index isomorphic to trivial group, $\ZZ$ or $\ZZ^2$ (as usual, $\{S_n\}_{n\ge0}$ is recurrent if $\sum_{n\ge0}\Prob_{e}(S_n=e)=\infty$; otherwise it is transient).

The main aim of this paper is to establish  strong law of large numbers (SLLN) and central limit theorem (CLT) for the capacity of the range process of  $\{S_n\}_{n\ge0}$.
Recall that the range process $\{\calR_n\}_{n\ge0}$ is defined as the sequence of random sets
  $$\calR_n = \{S_0,\dots,S_n\}.$$
For $1\le m\le n$ we use the notation $\calR[m,n]=\{S_m,\dots,S_n\}$,  $\calR[n,\infty)=$\linebreak $\{S_n,S_{n+1},\dots\}$ and $\calR_\infty=\{S_0,S_{1},\dots\}$.
The capacity of a set $A\subseteq\GG$ (with respect to the random walk $\{S_n\}_{n\ge0}$) is defined as 
\begin{align*}
{\rm Cap}(A)=\sum_{g\in A}\Prob_g(\tau^+_A=\infty).
\end{align*}
Here, $\tau^+_A$ denotes the first return time of $\{S_n\}_{n\ge0}$ to the set $A$, i.e.
\begin{align*}
\tau^+_A=\inf\{n\ge1:S_n\in A\}.
\end{align*} 
We are interested in the asymptotic behavior of the process  $\{\calC_n\}_{n\ge0}$ defined as 
  $$\calC_n={\rm Cap}(\calR_n ).$$

Before stating our main results, we first remind the reader of the notion of group growth rate.
For $g\in\GG$, define its word length $\rho(g)$ as
  $$\rho(g) = \inf \{m\ge0: g=\gamma_1\dots\gamma_m, \text{for }\gamma_1,\dots,\gamma_m\in \Gamma\}.$$
In other words, $\rho(g)$ is the minimal number of ``letters'' from ``alphabet'' $\Gamma$ needed to  express $g$.
The growth function $n\mapsto \mathcal{V}(n)$ of $\GG$ is defined as (as usual, $|S|$ denotes the size of set $S$)
  $$\mathcal{V}(n)= |\{g\in\GG:\rho(g)\le n\}|.$$
Although the growth function depends on the choice of $\Gamma$, its order of magnitude is independent of that choice. Of special interest are groups of polynomial growth, i.e.\ groups for which $\mathcal{V}(n)\asymp n^d$ for some integer growth index $d\ge0$. Groups for which $\mathcal{V}(n)\gtrsim n^d$ for all $d\geq 0$ are said to have superpolynomial growth and they have infinite growth index ($d=\infty$) by definition.\footnote{Here we use the standard asymptotic notation -- see the end of this section.} One can show that a finitely generated group has either polynomial or superpolynomial growth.\footnote{The content of this statement is in discarding the existence of groups of growth e.g. $n^{5/2}$ or $n^2\log n$.} Unsurprisingly, the growth index of $\ZZ^d$ is equal to $d$.  The celebrated theorems of Bass and Gromov show that a finitely generated group is of polynomial growth if and only if it is virtually nilpotent, see \cite[Theorem 2]{Bass-1972} and \cite[Main theorem]{Gromov-1981} (see also \cite[Theorems VI.2.1 and VI.2.2]{Varopoulos-Saloff-Coste-Coulhon}).


We now state the main results of the paper. Here, and in the rest of the paper, $d$ will always denote the growth index of the group of polynomial growth we are working on.

\begin{theorem}\label{SLLN} There is $\mu_d\ge0$ such that
	\begin{equation}\label{eq:SLLN}
	\lim_{n \to \infty}\frac{\calC_n }{n}=\mu_d,\qquad \Prob\text{-a.s.}
	\end{equation}
	Furthermore, $\mu_d=0$ if and only if  $d\le4$.
\end{theorem}

\begin{theorem}\label{CLT}
  Assume that  $d\ge6$.  There is  $\sigma_d>0$ such that
    \begin{equation}\label{eq:CLT}
	  \frac{\calC_n - \E \left[\calC_n\right] }{\sigma_d\,n^{1/2}}\xrightarrow[n\to\infty]{\rm{(d)}}   \calN(0, 1),
	\end{equation}
  where $\calN(0, 1)$ stands for the standard normal random variable and $\xrightarrow[n\to\infty]{\rm{(d)}}$ denotes the convergence in distribution.
\end{theorem}
\noindent We note that the restriction $d\ge6$ in Theorem \ref{CLT} is necessary if one is interested in a general result claiming that in every group with growth index $d$ there is a nondegenerate Gaussian limit under the scaling $n^{1/2}$. Indeed, for the symmetric simple random walk on $\ZZ^5$, \cite[Theorem B]{Schapira_5} shows that one requires $(n\log n)^{1/2}$ scaling in \eqref{CLT}. Additionally, for the symmetric simple random walk on $\ZZ^3$ or $\ZZ^4$, \cite[Theorem 1.2]{Chang} and \cite[Theorem 1.1]{Asselah_Z4} proved that  
 the limit in \eqref{eq:CLT} is non-Gaussian (case $d=4$ also requires $n/\log^2n$ scaling). Finally, from \cite[Theorems VI.2.1, VI.2.2, VI.3.3 and VI.5.1]{Varopoulos-Saloff-Coste-Coulhon} one easily concludes that $d\le2$ 
 if and only if the underlying random walk is recurrent. It is easy to see that in this case we have $\calC_n\equiv0$.

\subsection*{Literature overview and related results}

The study on the range process \linebreak $\{\calR_n\}_{n\ge0}$ of  $\ZZ^d$-valued random walks  has a long history. 

A pioneering work is due to Dvoretzky and Erd\H{o}s \cite{Dvoretzky} where they obtained a SLLN for $\{|\calR_n|\}_{n\ge0}$ in the context of the symmetric simple random walk in dimension $d\ge2$.  Their result was later extended by Spitzer  \cite{Spitzer} for an arbitrary random walk in dimension $d\ge1$.  CLT for $\{|\calR_n|\}_{n\ge0}$ was obtained by Jain and Orey \cite{Jain-Orey}  for so-called strongly transient random walks. Jain and Pruit \cite{Jain-Pruitt-1971} extended this result to all random walks in dimension $d\ge3$, while  Le Gall \cite{LeGall-1986} discussed  this problem in $\ZZ^2$.  Jain and Pruitt \cite{Jain-Pruitt.1972} also established a law of the iterated logarithm for $\{|\calR_n|\}_{n\ge0}$ in the case when the underlying random walk is either strongly transient or in dimension $d\ge4.$ Bass and Kumagai \cite{Bass-Kumagai-2002}  later extended this result to $\ZZ^2$ and $\ZZ^3$. 

The first results on the asymptotic behavior  of the capacity process $\{\calC_n\}_{n\ge0}$  are due to  Jain and Orey \cite{Jain-Orey} who obtained a version of the  SLLN for any transient random walk.
Very recently Asselah, Schapira and Sousi \cite{Asselah_Zd} proved  a CLT  for $\{\calC_n\}_{n\ge0}$ for the symmetric simple random walk in dimension $d\ge6$. Schapira \cite{Schapira_5} obtained an analogous result in  $\ZZ^5$, while  versions of  a SLLN and CLT in  $\ZZ^4$ and $\ZZ^3$
 were proved by Asselah, Schapira and Sousi in \cite{Asselah_Z4} and Chang in \cite{Chang}, respectively. 

On the other hand, to the best of our knowledge, there are only a few works on asymptotic behavior of the range process of random walks on  groups other than $\ZZ^d$. One of the first results in this direction is due to Chen, Yan and Zhou \cite{Chen-Yan-Zhou-1997} who obtained a SLLN and CLT for $\{|\calR_n|\}_{n\ge0}$ of the symmetric simple random walk on the $N$-ary tree with $N\ge2$ (i.e. the free product $\ZZ_2^{\ast N}$ with the standard $N$ generators). Recently, Benjamini, Izkovsky and Kesten \cite{Benjamini}
studied a  variant of the SLLN for $\{|\calR_n|\}_{n \ge 0}$ conditioned on the event $\{S_n=S_0\}$ (the so-called random walk bridge)
associated to a symmetric random walk on  a finitely generated group.
They  showed that the conditional and corresponding unconditional SLLN limits coincide. Finally, entropy properties of the range process of random walks on finitely and infinitely generated groups were studied by Chen,  Xie and Zhao \cite{Chen-Xie-Zhao}, Erschler \cite{Erschler} and Windsch \cite{Windsch}.

As we have already commented, in this paper we focus on the asymptotic behavior of the capacity process $\{\calC_n\}_{n\ge0}$ associated to a symmetric random walk on a finitely generated group. As the main results we obtain the SLLN and CLT stated in Theorems \ref{SLLN} and \ref{CLT}, and show that these types of behaviours, as in the symmetric simple random walk on $\ZZ^d$,  depend on the growth index $d$ only.

\subsection*{Notation}
Clearly, $\{S_n\}_{n\ge0}$ satisfies the (strong) Markov property (with respect to the corresponding natural filtration) with transition probabilities given by 
  $$p_n(g_1,g_2)=\Prob_{g_1}(S_n=g_2).$$
Note that $p_n$ is left-translation-invariant (i.e., $p_n(hg_1,hg_2)=p_n(g_1,g_2)$ for all $h\in\GG$) so we may consider $p_n$s as one-variable functions, i.e.\ $p_n(g_1^{-1}g_2)=p_n(e,g_1^{-1}g_2)=p_n(g_1,g_2)$.

We denote by 
\begin{align*}
\mathcal{G}(g_1,g_2)=\sum_{n\ge0}p_n(g_1,g_2)
\end{align*}
the Green function of $\{S_n\}_{n\ge0}$. Due to the aforementioned left-translation-invariance, we sometimes write $\mathcal{G}(g_1^{-1}g_2)$ instead of $\mathcal{G}(g_1,g_2)$. 

We will use the standard asymptotic notation. For $f,g:\N\to[0,\infty)$  we write $f(n) \lesssim g(n)$ if and only if there is a constant $c>0$  such that $f(n)\le c \,g(n)$ for all $n\in\N$. We would like to stress that, unless otherwise specified, implied constant $c$ will depend only on $\GG$ (most of the time only on its growth index $d$ and the set of generators $\Gamma$). In rare circumstances where it depends on some other parameters, we will indicate them by adding a subscript after $\lesssim$. We will write $f(n) \asymp g(n)$ if and only if $f(n) \lesssim g(n)$ and $g(n) \lesssim f(n)$. Finally, we use the standard $\mathsf{o}$ notation: for $f,g:\N\to[0,\infty)$  we write $f(n) =\mathsf{o}(g(n))$ if and only if $\lim_{n\to\infty}f(n)/g(n)=0$.

\subsection*{Organisation of the paper}
The rest of the paper is organised as follows. In Section \ref{S2}, we prove Theorem \ref{SLLN}. We first show that the limit in \eqref{eq:SLLN} exists and is finite, and then we show that it vanishes if and only if $d\le4$.  In Section \ref{S3},  we present minor modifications needed in the proof of \cite[Theorem 1.1]{Asselah_Zd} to prove Theorem \ref{CLT}. Roughly, we first  determine the constant $\sigma_d$ in \eqref{eq:CLT}, and then by employing the Lindeberg-Feller theorem we prove Theorem \ref{CLT}. At the end we also briefly discuss the functional version of the CLT in  \eqref{eq:CLT}.

\section{On the SLLN for $\{\calC_n\}_{n\ge0}$}\label{S2}

In this section, we prove  Theorem \ref{SLLN}. 
We start with the proof of \eqref{eq:SLLN}.

\begin{proposition}\label{Proposition 2.1} 
There is $\mu_d\ge0$ such that
$$
	\lim_{n \to \infty}\frac{\calC_n }{n}=\mu_d,\qquad \Prob\text{-a.s. and in } \mathrm{L}^1.
$$
\end{proposition}
\begin{proof} For $0\le m\le n$, let $$\mathcal{F}_{m,n}=\Capa(\calR[m,n]).$$ Clearly, the family $\{\mathcal{F}_{m,n}:0\le m\le n\}$ forms a subadditive process in the sense of \cite[Chapter 1]{Krengel}. Then, according to Kingman's subadditive ergodic theorem (see \cite[Chapter 1, Theorem 5.3]{Krengel}) it holds that $\{\calC_n/n\}_{n\ge1}$ converges a.s.\ and in $\mathrm{L}^1$ to a nonnegative integrable random variable. Further, for $0\le m\le n$,  using left-translation-invariance and subadditivity of the capacity (see \cite[Proposition 25.11]{Spitzer}) we obtain
  \begin{align*}\Capa(\{X_n^{-1}\cdots X_{m+1}^{-1},\dots,e\})&=\Capa(X_n^{-1}\cdots X_1^{-1}\calR[m,n])=\Capa(\calR[m,n])\\&\le \calC_n\le \calC_{m}+\Capa(\calR[m,n])\\
  &= \calC_{m}+\Capa(\{X_n^{-1}\cdots X_{m+1}^{-1},\dots,e\}).\end{align*}
Thus,
  $$\lim_{n \to \infty}\frac{\calC_n}{n}=\lim_{n \to \infty}\frac{\Capa(\{X_n^{-1}\cdots X_{m+1}^{-1},\dots,e\})}{n},$$
which shows that the limit is measurable with respect to the tail $\sigma$-algebra \linebreak $\bigcap_{m=1}^\infty\sigma(X_m, X_{m+1},\dots)$. According to Kolmogorov's $0$-$1$ law, this $\sigma$-algebra is trivial. The assertion now follows from \cite[Chapter 2, Lemma 2.9]{Kallenberg}.
		\end{proof}
As a direct consequence of Proposition \ref{Proposition 2.1} it follows that \begin{equation}\label{eqe}\lim_{n \to \infty}\frac{\E[\calC_n] }{n}=\mu_d.\end{equation}

We now turn to showing that the limit $\mu_d$ in \eqref{eq:SLLN} vanishes if and only if   $d\le4$. As we have already commented, when $d\le2$ it necessarily holds that $\calC_n\equiv0$. We first show several auxiliary results.

\begin{proposition}\label{prop2.x}
If $d=\infty$, then 
$$
\sup_{g\in\GG}p_n(g) \lesssim_q n^{-q}$$ for all $q\ge0$. If $d<\infty$, then $$
\sup_{g\in\GG}p_n(g) \lesssim n^{-d/2}\qquad \text{and}\qquad p_{2n}(e)\gtrsim n^{-d/2}.$$ 
\end{proposition}
\begin{proof}The assertions are a direct consequence of  \cite[Theorems VI.2.1, VI.2.2, VI.3.3 and VI.5.1]{Varopoulos-Saloff-Coste-Coulhon}.
\end{proof}

Further, we conclude the following.

\begin{proposition}\label{Proposition 2.3}
  We have
  $$\E\ugl{\sum_{k, \ell = 0}^n \mathcal{G}(S_k, S_\ell)} \lesssim \begin{cases}
		  n & d\ge5,\\
		  n\log n & d=4,\\
		  n^{3/2} & d=3.
	\end{cases}$$
\end{proposition}

\begin{proof}
  We have that
	\begin{align*}
	 \E\ugl{\sum_{k, \ell = 0}^n \mathcal{G}(S_k, S_\ell)}
	&= \E\ugl{\sum_{k = 0}^n \mathcal{G}(e) 
		+ \sum_{k = 0}^n \sum_{\substack{\ell = 0 \\ l \neq k}}^n \mathcal{G}(S_{\aps{k - \ell}})} \\
	&=\E\ugl{(n+1) \mathcal{G}(e) + 2\sum_{k = 0}^n  \sum_{\ell = 1}^{n-k} \mathcal{G}(S_\ell)} \\
		&=(n+1) \mathcal{G}(e) + 2\sum_{k = 0}^n  \sum_{\ell = 1}^{n-k} \sum_{g\in\GG}\mathcal{G}(g)\Prob(S_\ell=g) \\
		&=(n+1) \mathcal{G}(e) + 2\sum_{k = 0}^n  \sum_{\ell = 1}^{n-k} \sum_{g\in\GG}\sum_{m=0}^\infty p_m(e,g)p_\ell(g,e) \\
			&=(n+1) \mathcal{G}(e) + 2\sum_{k = 0}^n  \sum_{\ell = 1}^{n-k} \sum_{m=\ell}^\infty p_{m}(e).
			\end{align*}
  The statement of the proposition now follows using the upper bounds from Proposition \ref{prop2.x} and elementary summations.
\end{proof}

As a consequence of Proposition \ref{Proposition 2.3} we conclude the following.

\begin{proposition}\label{Proposition 2.4}
  We have
    $$\E \left[\calC_n\right] \gtrsim \begin{cases}
		  n & d\ge5,\\
		  n/\log n & d=4,\\
		  n^{1/2} & d=3.
		\end{cases}$$
\end{proposition}

\begin{proof}
	For fixed $n\ge1$ we consider the following (random) probability measure defined on $\GG$: 
	\begin{equation}\label{eq:def_of_nu}
	\nu_n(g) = \frac{1}{n} \sum_{k = 1}^{n} \delta_{S_k}(g).
	\end{equation}
	Clearly, $\mathrm{supp}\,\nu_n =\mathcal{R}[1,n]$.
	 According to (a straightforward modification of) \cite[Lemma 2.3]{Jain},  the capacity of a set $A\subseteq\GG$ has the following representation 
	 $$\mathrm{Cap}(A)=\frac{1}{\inf_\nu\sum_{g_1,g_2\in A}\mathcal{G}(g_1,g_2)\nu(g_1)\nu(g_2)},$$ 
	 where the infimum is taken over all probability measures on $\GG$ with $\mathrm{supp}\,\nu\subseteq A.$
Clearly, $$\E\left[\sum_{g_1, g_2 \in \mathcal{R}_n} \mathcal{G}(g_1, g_2) \nu_n(g_1)\nu_n(g_2)\right]=\frac{1}{n^2}\E\ugl{\sum_{k, \ell = 1}^n \mathcal{G}(S_k, S_\ell)}.$$ Recall that $\calC_n=\mathrm{Cap}(\mathcal{R}_n)$. The assertion now follows from Jensen's inequality and Proposition \ref{Proposition 2.3}.
\end{proof}

Notice that the previous proposition already proves the $d\ge5$ case of the second assertion of Theorem \ref{SLLN}. To cover the cases $d=3$ and $d=4$, we need to prove that in those settings we have $\E \left[\calC_n\right] = \mathsf{o}(n)$. 

For $r>0$,  let  $\mathscr{B}_\GG(r)=\{g\in\GG\colon \rho(g)<r\}$ and $\overline{\mathscr{B}}_\GG(r)=\{g\in\GG\colon \rho(g)\le r\}$   denote the open and closed ball around  $e$ of radius $r$, respectively. Further, let  $\tau_r=\inf\{n\ge0:S_n\in \mathscr{B}^c_\GG(r)\}$
\begin{proposition}\label{prop2.xx}
  Suppose $0<d<\infty$. There are $c,C>0$  such that
  \begin{enumerate}
    \item[(i)] $\displaystyle\Prob\bigl(\tau_r<n\bigr)\lesssim e^{-c r^2/n}$ for all $r\ge1$; 
    \item[(ii)] $\displaystyle p_n(g) \gtrsim  n^{-d/2}e^{-C\rho(g)^2/n}\mathbf{1}_{\mathscr{B}_\GG(cn)}(g)$ for all  
    $g\in\GG$;
    \item[(iii)] $\displaystyle p_n(g)\lesssim n^{-d/2}e^{-c\rho(g)^2/n}$ for all 
    $g\in\GG$.
  \end{enumerate}
\end{proposition}
\begin{proof}
First assertion follows from \cite[Theorem VI.5.1]{Varopoulos-Saloff-Coste-Coulhon}, \cite[Theorem 2.1]{Hebisch-Saloff-Coste-1993} and \cite[Proposition 4.33 and Theorem 5.23]{Barlow-Book-2017}, while the second and third follow from \cite[Theorems 2.1 and 5.1]{Hebisch-Saloff-Coste-1993}.
\end{proof}

We now deal with the case $d=3$.

\begin{proposition}\label{Proposition 2.6}
 If   $d=3$, then
		\begin{align*}
		\E \left[\calC_n\right] \asymp n^{1/2}.
		\end{align*}
\end{proposition}
\begin{proof}
  According to Proposition \ref{Proposition 2.4} it suffices to show that $\E \left[\calC_n\right]  \lesssim n^{1/2}$. We follow  \cite{Asselah_Zd} and \cite{Lawler-Book-2013}, where the case of the symmetric simple random walk on $\ZZ^d$  has been considered. However, we note that in \cite{Asselah_Zd} the case when $d=3$ was covered by invoking the fact that for the symmetric simple random walk on $\ZZ^d$, the process $\{\rho(S_n)\}_{n\geq 0}$ is a submartingale. Unfortunately, this is not true in general (see Remark~\ref{dead-end}). For this reason, our proof is somewhat more complicated and closer to the approach \cite{Asselah_Zd} used for $d=4$ case.
  
  Let $\{\widetilde S_n\}_{n\ge0}$ be an independent copy of $\{ S_n\}_{n\ge0}$, with distribution $\widetilde{\Prob}$.  Due to \cite[Section 2]{Jain-Orey} it holds that (with the convention $\calR_{-1} = \emptyset$) 
    $$\calC_n = \sum_{k=0}^n \widetilde\Prob_{S_k}\bigr( (S_k\,\widetilde\calR[1,\infty))\cap\calR_n=\emptyset \bigl) \mathbf{1}_{\{S_k\notin\calR_{k-1}\}},$$
  where $\{\widetilde\calR_n\}_{n\ge0}$ stands for the range process of $\{\widetilde S_n\}_{n\ge0}$. Next, for fixed $0\le k\le n$ we consider three independent random walks $\{S^1_i\}_{0\le i\le k}$, $\{S^2_i\}_{0\le i\le n-k}$ and $\{S^3_i\}_{i\ge 0}=\{\widetilde S_i\}_{i\ge 0}$, where $S_i^1=S_{k}^{-1} S_{k-i}$ and $S_i^2=S_k^{-1} S_{k+i}$. We now have
    $$\calC_n = \sum_{k=0}^n \widetilde\Prob\bigr(\calR^3[1,\infty)\cap(\calR^1_k\cup\calR^2_{n-k})=\emptyset\bigl)\mathbf{1}_{\{e\notin\calR^{1}[1,k]\}},$$
  where $\{\calR^1_i\}_{0\le i\le k}$, $\{\calR^2_i\}_{0\le i\le n-k}$ and $\{\calR^3_n\}_{n\ge0}$ stand for the range processes of \linebreak $\{S^1_i\}_{0\le i\le k}$, $\{S^2_i\}_{0\le i\le n-k}$ and $\{S^3_i\}_{i\ge 0}$, respectively. By taking expectations, we conclude 
  \begin{align*}
    \E \left[\calC_n\right]  &= \sum_{k=0}^n\Prob\bigr(e\notin\calR^{1}[1,k],\,\calR^3[1,\infty)\cap(\calR^1_k\cup\calR^2_{n-k})=\emptyset\bigl)\\
    &\le 2\sum_{k=0}^{n/2}\Prob\bigr(\calR^3[1,\infty)\cap(\calR^1_k\cup\calR^2_{n-k})=\emptyset\bigl)\\
    &\le 2\sum_{k=0}^{n/2}\Prob\bigr(\calR^3[1,k]\cap(\calR^1_k\cup\calR^2_k)=\emptyset\bigl).
  \end{align*}
  Hence, the desired bound will follow if we show that 
     \begin{equation}\label{eqF}\Prob\bigr(\calR^3[1,k]\cap(\calR^1_{k}\cup\calR^2_{k})=\emptyset\bigl) \lesssim k^{-1/2}.\end{equation}
		
  To prove this, we follow the   proof for the upper bound  in \cite[Theorem 3.5.1]{Lawler-Book-2013} (in the case of the symmetric simple random walk on $\ZZ^3$), given in \cite[Section 3.6]{Lawler-Book-2013}. Observe that the probability on the left-hand side in \eqref{eqF} corresponds to the function $F(k)$ in \cite[Theorem 3.5.1]{Lawler-Book-2013}. 
 A straightforward inspection shows that all the arguments in the proofs of \cite[Theorem 6.1 and Proposition 3.6.3]{Lawler-Book-2013} and the first half (up to relation (3.26)) of the proof of \cite[Proposition 3.6.2]{Lawler-Book-2013} hold also for finitely generated groups with symmetric set of generators and $d=3$. Hence, the only part that has to be clarified is the   relation (3.26), i.e.\ that 
  \begin{equation}\label{ocekivanje-Zn}
    \E[Z_n^{-1}] \lesssim n^{1/2}
  \end{equation}
  where $Z_n$ is defined in the following way.
 Let $\lambda \in[\frac12,1)$ and $n=\lfloor (1-\lambda)^{-1} \rfloor$, and let $\Lambda$ be a geometric random variable with parameter $1-\lambda$, independent of the random walk $\{S_n\}_{n\ge0}$. Set
    $$\mathcal{G}_{(\lambda)}(g)=\E\left[\sum_{k=0}^{\Lambda}\mathbf{1}_{\{S_k=g\}}\right] \qquad\text{and}\qquad Z_n=\inf_{0\le i\le j\le n}\mathcal{G}_{(\lambda)}(S_j^{-1}S_i).$$
 We now prove \eqref{ocekivanje-Zn}. Let $c,C>0$ be as in the second part of Proposition \ref{prop2.xx}, and notice that for $g\in\GG$ such that $\rho(g)>32/c^2$ we have 
    $$\sum_{k=c\rho(g)^2/32}^{c\rho(g)^2/16}p_k(g) \gtrsim \sum_{k=c\rho(g)^2/32}^{c\rho(g)^2/16} k^{-3/2}e^{-C\rho(g)^2/k} \gtrsim \rho(g)^{-1}.$$
  This provides us, for such $g$, the following lower bound for $\mathcal{G}_{(\lambda)}(g)$:
    $$\mathcal{G}_{(\lambda)}(g)\ge \Prob(\Lambda\ge c\rho(g)^2/16)\sum_{k=0}^{c\rho(g)^2/16}p_k(g)\gtrsim e^{-c\rho(g)^2(1-\lambda)/8} \rho(g)^{-1},$$
where we used the elementary inequality $\lambda\ge e^{-2(1-\lambda)}$ (recall that $\lambda\in[\frac{1}{2},1]$). Clearly, since the set $\{g\in\GG:0<\rho(g)\le 32/c^2\}$ is finite and $\mathcal{G}_{(\lambda)}(g) \geq \mathcal{G}_{(\frac12)}(g)>0$, we conclude that
    $$\mathcal{G}_{(\lambda)}(g)\gtrsim \rho(g)^{-1}e^{-c\rho(g)^2(1-\lambda)/8}$$
  for all $g\in\GG\setminus\{e\}$ and $\lambda\in[\frac12,1]$.
  
  Now, as in \cite[Proposition 3.6.2]{Lawler-Book-2013} let
    $$R_n=\sup_{0\le i\le j\le n}\rho(S_j^{-1}S_i)$$
  and note that
    $$Z_n^{-1}\lesssim R_ne^{cR_n^2/(8n)}=f(R_n),$$
  where $f(u)=ue^{cu^2/(8n)}$. Since $\Prob(R_n \geq r)\le \Prob(\sup_{0\le i\le n}\rho(S_i) \geq r/2)$ and $R_n \le n$, we have that
  \begin{align*}
    \E[Z_n^{-1}] &\lesssim \int_0^n f'(r)\Prob(R_n \geq r)\, dr\\&\lesssim \sum_{k=1}^ne^{c k^2/(8n)}\left(1+\frac{ck^2}{4n}\right)\Prob\bigl(\sup_{0\le i\le n}\rho(S_i) \geq k/2\bigr).
  \end{align*}
  Further, observe that 
    $$\Prob\bigl(\sup_{0\le i\le n}\rho(S_i)\geq k/2\bigr)= \Prob\bigl(\tau_{k/2}<n+1\bigr).$$
  Thus, first part of Proposition~\ref{prop2.xx} implies that
  \begin{align*}
    \E[Z_n^{-1}]&\lesssim \sum_{k=1}^ne^{c k^2/(8n)}\left(1+\frac{ck^2}{4n}\right)e^{-ck^2/(4n)}\\
    &= \sum_{k=1}^ne^{-c k^2/(8n)}\left(1+\frac{ck^2}{4n}\right) \lesssim n^{1/2},
  \end{align*}
  which concludes the proof.
\end{proof}

\begin{remark}\label{dead-end}
  As we mentioned in the previous proof, $\{\rho(S_n)\}_{n\geq 0}$ is a submartingale in the case of the symmetric simple random walk on $\ZZ^d$. This is not necessarily the case in general.   For example, one might even have that $\rho(g\gamma)\leq \rho(g)$ for some $g\in\GG$ and all $\gamma\in\Gamma$, with strict inequality occurring for at least one $\gamma$. This rules out the possibility of $\{\rho(S_n)\}_{n\geq 0}$ being a submartingale since the submartingale inequality $\E[\rho(S_{n+1}) | X_1,\dots,X_n]\geq \rho(S_n)$ would obviously not be satisfied on the event $\{S_n=g\}$. The simplest example we are aware of that demonstrates this phenomenon is provided by the group $\ZZ$ with the set of generators $\{\pm 2,\pm 3\}$ and $g=1$.
\end{remark}

We now consider the case $d=4$.  
For $k\in\N$, let
$$V_k=(\mathscr{B}_\GG(2^{k})\setminus\mathscr{B}_\GG(2^{k-1}))\cap \calR_{\tau_{2^{k+1}}}.$$

\begin{proposition}\label{lm:harnack}
  Let $\{\widetilde{S}_n\}_{n\ge0}$ be an independent copy of $\{S_n\}_{n\ge0}$, and let $E_k=\{V_k\cap \widetilde{V}_k\neq\emptyset\}$, where $\widetilde V_k$ is defined as $V_k$ but in terms of $\{\widetilde{S}_n\}_{n\ge0}$. There is $c>0$ such that for all $\ell\in\N$ it holds that \begin{equation}\label{harnack}
    \sup_{(g_1,g_2)\in \overline{\mathscr{B}}_\GG(2^{4\ell-3})\times \overline{\mathscr{B}}_\GG(2^{4\ell-3})}\Prob_{(g_1,g_2)}(E_{4\ell})\le c\,\inf_{(g_1,g_2)\in \overline{\mathscr{B}}_\GG(2^{4\ell-3})\times \overline{\mathscr{B}}_\GG(2^{4\ell-3})}\Prob_{(g_1,g_2)}(E_{4\ell}).
  \end{equation} 
\end{proposition}

\begin{proof}
Notice that we may assume that $\ell$ is large enough when needed. The idea is to use the elliptic Harnack inequality given in \cite[Theorem 7.18]{Barlow-Book-2017}, and for that goal we first need to set the stage by making a couple of group theoretic arguments.

Let $\mathsf{H}=\langle\Gamma\times\Gamma\rangle$.  Note that $\mathsf{H}$ can be a proper subgroup of $\GG\times\GG$. For example, if $\GG=\ZZ$ and $\Gamma=\{-1,1\}$, then $\mathsf{H}=\langle\Gamma\times\Gamma\rangle$ consists of elements of $\ZZ^2$ with coordinates of the same parity. This is actually typical -- it is easy to prove that either $\mathsf{H}=\GG\times\GG$ or $|\GG\times\GG : \mathsf{H}|=2$, i.e.\ in any case $\GG\times\GG$ can be covered by at most two translates of $\mathsf{H}$. Indeed, fix $\gamma\in\Gamma$ and take $(g,h)\in\GG\times\GG$. Let 
  $$g=g_1\dots g_r, \quad\text{and}\quad \gamma^{-r}h = h_1\dots h_s, \qquad g_1,\dots,g_r,h_1,\dots,h_s\in\Gamma.$$
Consider the following product of $r+s$ elements from $\Gamma\times\Gamma$,
  $$(g_1,\gamma)\dots(g_r,\gamma)  (\gamma,h_1)(\gamma^{-1},h_2)(\gamma,h_3)\dots(\gamma^{(-1)^{s+1}},h_s)\in\mathsf{H}.$$
This product, depending on the parity of $s$, will be either $(g,h)$ or $(g\gamma,h)$. Thus, $\GG\times\GG = \mathsf{H}\cup(\mathsf{H}\cdot(\gamma^{-1},e))$.

The fact that $|\GG\times\GG : \mathsf{H}|<\infty$ implies that $d_{\mathsf{H}}=d_{\GG\times\GG}$ (see \cite[Proposition 8.78]{drutu-kapovich}), where $d_{\mathsf{H}}$ and $d_{\GG\times\GG}$ denote the growth indices of $\mathsf{H}$ and $\GG\times\GG$, respectively. We now show that $d_{\GG\times\GG}=2d$, i.e.\ $d_\mathsf{H}=2d.$ Namely, as we have already commented, the growth index does not depend on the choice of the set of generators. Hence, we can
     consider the set  $\{(\gamma,e), (e,\gamma) \colon \gamma\in\Gamma\}$ of  $2|\Gamma|$ generators. It is straightforward to check that $\rho_{\GG\times\GG}(g,h) = \rho_{\GG}(g) + \rho_{\GG}(h)$. This in particular implies that for every $r>0$, $\mathscr{B}_{\GG}(r)\times \mathscr{B}_{\GG}(r)\subseteq\mathscr{B}_{\GG\times\GG}(2r)$ and  $\mathscr{B}_{\GG\times\GG}(r)\subseteq\mathscr{B}_{\GG}(r)\times \mathscr{B}_{\GG}(r)$, which proves the assertion. 
     
We next analyse the function $\rho_\mathsf{H}$, the word metric in $\mathsf{H}$. We show that there is $m\in\N$ such that 
\begin{equation}\label{metrike}
           \max(\rho(g),\rho(h)) + m\geq \rho_\mathsf{H}((g,h)) \geq \max(\rho(g),\rho(h))  
         \end{equation}
for all $(g,h)\in\mathsf{H}$. Since the second inequality is trivial, we concentrate on the first. 
Without loss of generality assume $\rho(g)\leq \rho(h)$, and let the corresponding shortest words be
  $$g=g_1\dots g_s\quad\text{and}\quad h=h_1\dots h_r,\qquad g_1,\dots,g_s,h_1,\dots,h_r\in\Gamma.$$
Let $\gamma\in\Gamma$ be arbitrary, and consider the product
  $$(g_1,h_1)(g_2,h_2)\dots (g_s,h_s) (\gamma,h_{s+1})(\gamma^{-1},h_{s+2})(\gamma,h_{s+3})\dots (\gamma^{(-1)^{r-s+1}},h_r) \in\mathsf{H}.$$
This product is either $(g,h)$ or $(g\gamma,h)$. In the latter case, write $(\gamma^{-1},e)$ as a product of some elements (say, $m$ of them) from $\Gamma\times\Gamma$ (this is possible since $(\gamma^{-1},e) = (g\gamma,h)^{-1}(g,h)\in\mathsf{H}$) and concatenate  it with the above product. This gives a representation of $(g,h)$ as a product of at most $r+m$ elements from $\Gamma\times\Gamma$, hence  \eqref{metrike}.

Fix some $\ell\in\N$, and define $f_\ell:\mathscr{B}_\GG(2^{4\ell-1})\times \mathscr{B}_\GG(2^{4\ell-1})\to[0,1]$ by
  $$f_\ell(g_1,g_2)=\Prob_{(g_1,g_2)}(E_{4\ell}).$$
From \eqref{metrike} we have that $\overline{\mathscr{B}}_\GG(2^{4\ell-3})\times\overline{\mathscr{B}}_\GG(2^{4\ell-3})\subseteq\mathscr{B}_{\mathsf{H}}(2^{4\ell-3}+m+1)$. We aim to apply the elliptic Harnack inequality for elements from this ball, and for this we need $f_\ell$ to be harmonic on $\mathscr{B}_{\mathsf{H}}(2^{4\ell-2}+2m+2)$. However,  \eqref{metrike} implies also that $\mathscr{B}_{\mathsf{H}}(2^{4\ell-2}+2m+2)\subseteq\mathscr{B}_{\GG}(2^{4\ell-2}+2m+2)\times \mathscr{B}_{\GG}(2^{4\ell-2}+2m+2)$. For $\ell$ large enough the latter set is contained in $\mathscr{B}_\GG(2^{4\ell-1})\times\mathscr{B}_\GG(2^{4\ell-1})$, and the Markov property easily shows that $f_\ell$ is harmonic on this set.
 
Finally, it is clear that $\{(S_n,\widetilde S_n)\}_{n\ge0}$ is a random walk on $\mathsf{H}$ with transition probabilities $$p^\mathsf{H}_n((g_1,g_2),(h_1,h_2))=p_n(g_1,h_1)p_n(g_2,h_2).$$ 
Using the fact that $d_\mathsf{H}=2d=8$, and employing  \cite[Corollary 7.17 and Theorem 7.18]{Barlow-Book-2017} together with Proposition \ref{prop2.xx}  for function $f_\ell$ gives the inequality \eqref{harnack}.
\end{proof}

We now conclude the following.

\begin{proposition}\label{Proposition 2.8}
 If $d=4$, then
  \begin{align*}
    \E \left[\calC_n\right] =\mathsf{o}(n).
  \end{align*}
\end{proposition}

\begin{proof}
  As in Proposition \ref{Proposition 2.6}, we have that
  \begin{align*}
    \E \left[\calC_n\right]  &= \sum_{k=0}^n\Prob\bigr(e\notin\calR^{1}[1,k],\,\calR^3[1,\infty)\cap(\calR^1_k\cup\calR^2_{n-k})=\emptyset\bigl)\\
    &\le \sum_{k=0}^n\Prob\bigr(\calR^3[1,\infty)\cap\calR^1_k=\emptyset\bigl).
  \end{align*}
  Since
    $$\lim_{n\to\infty}\Prob\bigr(\calR^3[1,\infty)\cap\calR^1_{n}=\emptyset\bigl)=\Prob\bigr(\calR^3[1,\infty)\cap\calR^1_{\infty}=\emptyset\bigl),$$
  the claim would easily follow if we prove that the probability on the right hand side is equal to $0$.
We will show a somewhat stronger claim -- that $\calR^1_{\infty}$ is almost surely a recurrent set for $\{S_n^3\}_{n\geq 0}$, i.e.\ that
\begin{equation}\label{rec}
  \Prob(S^3_n\in\calR^1_{\infty}\ \text{for infinitely many } n) = 1.
\end{equation}
For $k\in\N$, let \begin{align*}
    \tau_{2^k}^1&=\inf\{n\ge0:S^1_n\in\mathscr{B}^c_{\mathsf{G}}(2^k)\},\\
    \tau_{2^k}^2&=\inf\{n\ge0:S^2_n\in\mathscr{B}^c_{\mathsf{G}}(2^k)\},\\
    V_k^1&=(\mathscr{B}_\GG(2^{k})\setminus\mathscr{B}_\GG(2^{k-1}))\cap \calR^1_{\tau^1_{2^{k+1}}},\\
     V_k^2&=(\mathscr{B}_\GG(2^{k})\setminus\mathscr{B}_\GG(2^{k-1}))\cap \calR^2_{\tau^2_{2^{k+1}}},\\
     E_k&=\{V_k\cap \widetilde{V}_k\neq\emptyset\}, 
    \end{align*} where $\widetilde V_k$ is defined as $V_k$ (see Proposition \ref{lm:harnack}) but in terms of an independent copy $\{\widetilde{S}_n\}_{n\ge0}$ of $\{S_n\}_{n\ge0}$.
By completely the same reasoning as in \cite[Theorem 6.5.10]{Lawler-Limic-Book-2010}, \eqref{rec} will follow if we show that $\Prob(E_k\ \text{i.o.})>0$. According to  \cite[Corollary A.6.2]{Lawler-Limic-Book-2010}, it suffices to show that
\begin{itemize}
    \item [(i)] $\displaystyle\sum_{k=1}^\infty\Prob(E_{4k})=\infty;$
    \item[(ii)] there is $c>0$ such that $\Prob(E_{4k}\cap E_{4\ell}) \leq c\, \Prob(E_{4k})\Prob(E_{4\ell})$ for all $k,\ell\in\N$.
\end{itemize}

The relation in (i) will follow easily once we prove $\Prob(E_{4k})\gtrsim 1/k$ using the second moment method. For $k\in\N$ set $\mathscr{A}_\GG(k)=\mathscr{B}_\GG(2^{k})\setminus\mathscr{B}_\GG(2^{k-1}).$
Expressing $|V_k^1\cap V_k^3|$ as a sum of indicator functions and taking expectation gives
\begin{equation}\label{eq:Ecardsum}\begin{aligned}
  \E\left[|V_k^1\cap V_k^3|\right] = \sum_{g\in\mathscr{A}_\GG(k)}\Prob(S^1_n=g\ &\text{for some } n<\tau^1_{2^{k+1}}) \\&\Prob(S^3_n=g\ \text{for some } n<\tau^3_{2^{k+1}}).
\end{aligned}\end{equation}
According to \cite[Theorem 1.31]{Barlow-Book-2017}, for all $k\in\N$ and $g\neq e$, we have
  $$\Prob\bigl(S^1_n=g\ \text{for some } n<\tau^1_{2^{k+1}}\bigr)=c(k)\sum_{n=0}^\infty\Prob\bigl(S_n^1=g,\ n<\tau_{2^{k+1}}^1\bigr),$$
where 
  $$c(k)^{-1} = \sum_{n=0}^\infty\Prob_g(S_n^1=g,\ n<\tau_{2^{k+1}}^1) \leq \sum_{n=0}^\infty\Prob_g(S_n^1=g) = \mathcal{G}(e).$$
On the other hand, due  to Proposition \ref{prop2.xx} and \cite[Theorems 4.25 and 4.26]{Barlow-Book-2017}, we obtain
  $$\Prob\bigl(S^1_n=g\ \text{for some } n<\tau^1_{2^{k+1}}\bigr)\ge \mathcal{G}(e)^{-1}\sum_{n=0}^\infty\Prob\bigl(S_n^1=g,\ n<\tau_{2^{k+1}}^1\bigr)\gtrsim \rho(g)^{-2}.$$
Plugging this into \eqref{eq:Ecardsum}, and using 
that $\rho(g)\le 2^{k}$ on $\mathscr{A}_\GG(k)$ and $|\mathscr{A}_\GG(k)|\asymp 2^{4k}$ (recall that $d=4$),
we now have
  $$\E\left[|V_k^1\cap V_k^3|\right] \gtrsim \sum_{g\in\mathscr{A}_\GG(k)}2^{-4k}\asymp 1,$$
and this completes the first part of the second moment method, i.e.\ proving that the appropriate first moment is large. It remains to prove that the second moment is small. We have that
\begin{equation}\begin{aligned}\label{eq:bound57}
  \E \left[|V_k^1\cap V_k^3|^2\right] = \sum_{g_1,g_2\in\mathscr{A}_\GG(k)}&\Prob(S^1_n=g_1\ \text{and}\ S^1_m=g_2\ \text{for some } n,m<\tau^1_{2^{k+1}})\\
  &\Prob(S^3_n=g_1\ \text{and}\ S^3_m=g_2\ \text{for some } n,m<\tau^3_{2^{k+1}}).
\end{aligned}\end{equation}
Notice that
\begin{equation}
\begin{aligned}
\label{eq:niz}
  &\Prob(S^1_n=g_1\ \text{and}\ S^1_m=g_2\ \text{for some } n,m<\tau^1_{2^{k+1}})\\
  &\quad \le 2\,\Prob(S^1_n=g_1\ \text{and}\ S^1_m=g_2\ \text{for some } n\le m<\tau^1_{2^{k+1}})\\
  &\quad \le 2\,\Prob(S^1_n=g_1\ \text{and}\ S^1_m=g_2\ \text{for some } n\le m)\\
  &\quad \le 2\,\sum_{m\ge n\ge 0} \Prob\bigl(S^1_n=g_1\bigr)\Prob_{g_1}\bigl(S^1_{m-n}=g_2\bigr)\\&\quad = 2\, \mathcal{G}(g_1) \mathcal{G}(g_1^{-1}g_2)
\end{aligned}\end{equation}
We can easily bound the Green function by using the bound $p_n(g)\lesssim n^{-2}e^{-c\rho(g)^2/n}$ provided by Proposition \ref{prop2.xx}. For $g\neq e$ this gives
  $$\mathcal{G}(g) \lesssim \sum_{n=1}^\infty n^{-2}e^{-c\rho(g)^2/n} \lesssim \int_1^\infty u^{-2}e^{-c\rho(g)^2/u}\,du \lesssim \rho(g)^{-2} \lesssim (1+\rho(g))^{-2},$$
where the last step was done so that we can also include $g=e$. Using this bound, from \eqref{eq:bound57} and \eqref{eq:niz} we obtain
\begin{align*}
  \E \left[|V_k^1\cap V_k^3|^2\right] &\lesssim \sum_{g_1,g_2\in\mathscr{A}_\GG(k)}2^{-4k}\bigl(1+\rho(g_1^{-1}g_2)\bigr)^{-4}\\
 &= \sum_{g_1\in\mathscr{A}_\GG(k)}2^{-4k}\sum_{g_2\in\mathscr{A}_\GG(k)}\bigl(1+\rho(g_1^{-1}g_2)\bigr)^{-4}.
\end{align*}
Since $|\mathscr{A}_\GG(k)| \asymp 2^{4k}$ and $\rho(g_1^{-1}g_2)<2^{k+1}$ whenever $g_1,g_2\in \mathscr{A}_\GG(k)$, this boils down to a sum we can easily deal with via dyadic decomposition:
\begin{align*}
  \E \left[|V_k^1\cap V_k^3|^2\right] &\lesssim \sum_{g\in\mathscr{B}_\GG(2^{k+1})}\bigl(1+\rho(g)\bigr)^{-4} \lesssim \sum_{\ell=0}^{k}\sum_{g\in\mathscr{A}_\GG(\ell+1)}\bigl(1+\rho(g)\bigr)^{-4}\\
&\lesssim \sum_{\ell=0}^{k} (2^{\ell+1})^4\bigl(1+2^\ell\bigr)^{-4} \lesssim k.\end{align*}
We are finally in a position to use a variant of the second moment method (see e.g. \cite[Lemma A.6.1]{Lawler-Limic-Book-2010})
\begin{align*}
  \Prob\bigl(E_k\bigr)&=\Prob\bigl( |V_k^1\cap V_k^3| >0 \bigr) =\lim_{u\to 0}\Prob\bigl(|V_k^1\cap V_k^3| \ge u\,\E\left[|V_k^1\cap V_k^3|\right]\bigr)\\
  &\ge\lim_{u\to 0}\frac{(1-u)^2 \bigl(\E\left[|V_k^1\cap V_k^3|\right]\bigr)^2}{\E\left[|V_k^1\cap V_k^3|^2\right]} \gtrsim \frac{1}{k},
\end{align*}
which proves (i). 

It remains to prove (ii). Assume $k<l$ and note that
\begin{equation}\label{eq:harn}
  \Prob\bigl(E_{4\ell}\vert E_{4k}\bigr) = \sum_{F}
\Prob\bigl(E_{4\ell}\vert E_{4k}\cap F\bigr)\,\Prob\bigl(F\vert E_{4k}\bigr),
\end{equation}
where sum goes over all possible trajectories of $\{S_n^3\}_{0\leq n\leq \tau^3_{2^{4k+1}}}$. Additionally, let $\calR_F$ be the set of all elements in $\GG$ visited by $\{S_n^3\}_{0\leq n\leq \tau^3_{2^{4k+1}}}$ that followed a fixed trajectory $F$, $h=S^3_{\tau^3_{2^{4k+1}}}$, and $H_g$ be the event in which $g$ is the first element of $\calR_F\cap V_{4k}^3$ visited by $\{S_n^1\}_{n\ge0}$. Then using Markov property we obtain
\begin{align*}
  \Prob\bigl(E_{4\ell}\vert E_{4k}\cap F\bigr) &= \sum_{g\in \calR_F\cap V_{4k}^3} \Prob\bigl(E_{4\ell}\vert H_g\cap E_{4k}\cap F\bigr)\, \Prob\bigl(H_g\vert E_{4k}\cap F\bigr)\\
  &= \sum_{g\in \calR_F\cap V_{4k}^3} \Prob_{(g,h)}\bigl(E_{4\ell})\, \Prob\bigl(H_g\vert E_{4k}\cap F\bigr).
\end{align*}
From Proposition \ref{lm:harnack} it follows that there is $c>0$ such that for all $\ell\in\N$ we have $\Prob_{(g_1,g_2)}(E_{4\ell})\le c\,\Prob_{(e,e)}(E_{4\ell})$ for all $(g_1,g_2)\in \overline{\mathscr{B}}_\GG(2^{4\ell-3})\times \overline{\mathscr{B}}_\GG(2^{4\ell-3})$.
Thus, from Markov property we can conclude that 
  $$\Prob\bigl(E_{4\ell}\vert E_{4k}\cap F\bigr) \leq c\,\Prob_{(e,e)}(E_{4\ell}),$$
and plugging this back into \eqref{eq:harn} we obtain
  $$\Prob\bigl(E_{4\ell}\vert E_{4k}\bigr)\le c\,\Prob(E_{4\ell}),$$
which proves (ii). 
\end{proof}

\begin{remark}
\begin{itemize}
    \item [(i)] We remark that in \cite[Corollary 1.4]{Asselah_Zd} a much more quantitative version of Proposition~\ref{Proposition 2.8} has been proven -- namely that for the symmetric simple random walk on $\ZZ^4$ one has $\E \left[\calC_n\right] \asymp n/\log n$. We conjecture, and leave as an open problem, that the same holds in the setting of finitely generated groups with symmetric set of generators and $d=4$.
    
    We believe there are couple of special cases of this conjecture that are interesting on its own. Most obviously, is $\E \left[\calC_n\right] \asymp n/\log n$ whichever symmetric set of generators of $\ZZ^4$ one considers?
    
    One could also consider the same question on a specific group with the growth index $4$ which does not contain $\ZZ^4$ as a finite index subgroup. An important example of such a group is the discrete Heisenberg group $H_3(\ZZ)$, i.e.\ the group of matrices of the form
      $$\begin{pmatrix}
         1 & a & c\\
         0 & 1 & b\\
         0 & 0 & 1\\
        \end{pmatrix}, \qquad a,b,c\in\ZZ.$$
    Is it true that $\E \left[\calC_n\right] \asymp n/\log n$ for the symmetric simple random walk on $H_3(\ZZ)$ with respect to the standard generators $\{x,y,x^{-1},y^{-1}\}$ where
      $$x = \begin{pmatrix}
         1 & 1 & 0\\
         0 & 1 & 0\\
         0 & 0 & 1\\
        \end{pmatrix}, \qquad
        y = \begin{pmatrix}
         1 & 0 & 0\\
         0 & 1 & 1\\
         0 & 0 & 1\\
        \end{pmatrix}?$$
    \item[(ii)] By completely the same reasoning as in Proposition \ref{Proposition 2.8} one can show that $
		\E \left[\calC_n\right] =\mathsf{o}(n)$ in the case when $d=3$. However, Proposition \ref{Proposition 2.6} gives much sharper result.
\end{itemize}
\end{remark}

 We finally have:

\begin{proof}[Proof of Theorem \ref{SLLN}]
  The assertions follow from Propositions \ref{Proposition 2.1}, \ref{Proposition 2.4}, \ref{Proposition 2.6} and \ref{Proposition 2.8} and relation \eqref{eqe}.
\end{proof}

\section{CLT for $\{\calC_n\}_{n\ge0}$}\label{S3}

In this section, we show how the proof of \cite[Theorem 1.1]{Asselah_Zd} can be adapted to prove Theorem \ref{CLT}. 
We first determine the constant $\sigma_d$ in \eqref{eq:CLT}. If $d\ge6$, we show that the sequence $\{\Var(\calC_n) / n\}_{n \ge 1}$ is convergent and its limit is exactly $\sigma_d^2$. 

We start with the following auxiliary result. For $n\ge0$,
we write $\mathcal{G}_n(g_1, g_2)$ for the Green function up to time $n$, i.e.
\begin{equation*}
  \mathcal{G}_n(g_1, g_2) = \sum_{k = 0}^n p_k(g_1,g_2).
\end{equation*}
\begin{proposition}\label{lm:bounds_on_GGG}
  Assume  $d \ge 5$. Then, for all $g \in \GG$,
  \begin{equation}\label{eq:bound_with_h_d}
	\sum_{g_1,g_2\in\GG} \mathcal{G}_n( g_1) \mathcal{G}_n( g_2) \mathcal{G}(g_1^{-1}gg_2) \lesssim \mathcal{E}(n),
  \end{equation}
  where 
    $$\mathcal{E}(n) = 
	\begin{cases}
	1, &\quad  d \ge 7,\\
	\log n, & \quad  d =6,\\
	\sqrt{n}, & \quad  d =5.\\
	\end{cases}$$
\end{proposition}

\begin{proof}
  For all $k, j \geq 0$ we have that
  \begin{align*}
	\sum_{g_1,g_2\in\GG } 
	p_k(e, g_1) p_j(e, g_2)  \mathcal{G}(g_1, gg_2 ) &= \sum_{i = 0}^{\infty} \sum_{g_1,g_2\in\GG } p_k(e, g_1) p_j(g, gg_2 ) p_i(g_1, gg_2) \\
	& = \sum_{i = 0}^{\infty} \sum_{g_1,g_2\in\GG} p_k(e, g_1) p_i(g_1, gg_2) p_j(gg_2, g) \\
	&= \sum_{i = 0}^{\infty} p_{k + i + j}(g).
  \end{align*}
  Using Proposition \ref{prop2.x}, we thus get
  $$\sum_{g_1,g_2\in\GG } 
	p_k(e, g_1) p_j(e, g_2)  \mathcal{G}(g_1, gg_2 ) \lesssim (j+k)^{-d/2+1}.$$
  Summing over $j,k=0,\dots,n$ gives the stated bounds.
\end{proof}

We now conclude the announced first step of the proof of the CLT.

\begin{proposition}\label{prop3.x}
Assume $d\ge6$. Then,  $\{\Var(\calC_n)/n\}_{n\ge1}$ converges to a finite limit.  \end{proposition}
\begin{proof} The proof proceeds analogously as for the symmetric simple random walk on $\ZZ^d$, $d\ge6$, given in \cite{Asselah_Zd}. 
For the reader's convenience we give a complete proof. For $n,m\in\N$, let  $\calR_n^{(1)}=\{S_n^{-1}S_0,\dots,S_n^{-1}S_n\}$, $\calR_{n,m}^{(2)}=\{S_{n}^{-1}S_n,\dots,S_n^{-1}S_{n+m}\}$, $\calC_n^{(1)}=\mathrm{Cap}(\calR_n^{(1)})$ and $\calC_{n,m}^{(2)}=\mathrm{Cap}(\calR_{n,m}^{(2)})$. Clearly, $\calR_n^{(1)}$ and $\calR_{n,m}^{(2)}$ ($\calC_n^{(1)}$ and $\calC_{n,m}^{(2)}$) are independent and have the same law as $\calR_n$ and $\calR_m$ ($\calC_n$ and $\calC_m$), respectively.
According to 
\begin{itemize}
    \item [(i)] capacity decomposition shown in \cite[Proposition 25.11]{Spitzer} and \cite[Proposition 1.2]{Asselah_Zd}:
	for any finite sets $A,B\subset\GG$, $$\mathrm{Cap}(A)+\mathrm{Cap}(B)-2\mathcal{G}(A,B)\le\mathrm{Cap}(A\cup B)\le \mathrm{Cap}(A)+\mathrm{Cap}(B)-\mathrm{Cap}(A\cap B);$$
    \item [(ii)] left-translation-invariance of the capacity: for any $n,m\in\N$, $$\calC_{n+m}=\mathrm{Cap}\bigl(S_n^{-1}\calR_{n+m}\bigr)=\mathrm{Cap}\bigl(\calR_n^{(1)}\cup\calR_{n,m}^{(2)}\bigr),$$ 
\end{itemize}
we conclude
\begin{equation}\label{decomp}
\calC_n^{(1)}+\calC_{n,m}^{(2)}-2\mathcal{G}(\calR_n^{(1)},\calR_{n,m}^{(2)})\le \calC_{n+m}\le \calC_n^{(1)}+\calC_{n,m}^{(2)}.\end{equation} 
We remark that both \cite[Proposition 25.11]{Spitzer} and \cite[Proposition 1.2]{Asselah_Zd} are stated only for the symmetric simple random walk on $\ZZ^d$, but the proofs are valid for any finitely generated group with symmetric set of generators.

Further,  for $n,m\in\N$ define	 $\overline{\calC}_n = \calC_n - \E[\calC_n]$, and similarly  $\overline{\calC}_n^{(1)}$ and     $\overline{\calC}_{n,m}^{(2)}$. Taking expectation in \eqref{decomp} and then subtracting those two relations yields
	\begin{equation*}
	\big\vert \overline{\calC}_{n+m} - (\overline{\calC}_n^{(1)} + \overline{\calC}_{n,m}^{(2)})\big\vert \le 2 \max \{\mathcal{G}(\calR_n^{(1)},\calR_{n,m}^{(2)}), \E[\mathcal{G}(\calR_n^{(1)},\calR_{n,m}^{(2)})]\}.
	\end{equation*}
	Proceeding in a completely analogous way as in \cite[Lemma 3.2]{CSS19} (and using \eqref{eq:bound_with_h_d}) we obtain 
\begin{equation}\label{eq:err}
	\E\left[\big(\mathcal{G}(\calR^{(1)}_n, \calR^{(2)}_{n,n})\big)^q\right] \lesssim_q \mathcal{E}(n)^q
\end{equation} 
    for every $q \in \bbN$. Recall the definition of the function $\mathcal{E}(n)$ from Proposition \ref{lm:bounds_on_GGG}.
		Denote  $\lVert\cdot\rVert_2=\E[(\cdot)^2]^{1/2}$. Clearly, $\Var(\calC_n)=\lVert\overline{\calC}_n\rVert_2^2$ for $n\in\N$.
The triangle inequality and independence of $\overline{\calC}_n^{(1)}$ and $\overline{\calC}_{n,m}^{(2)}$ together with  the estimate $\E[\mathcal{G}(\calR_n^{(1)},\calR_{n,m}^{(2)})] \le \Vert \mathcal{G}(\calR_n^{(1)},\calR_{n,m}^{(2)})\Vert_2$ and \eqref{eq:err} imply
	\begin{align*}
	\Vert\overline{\calC}_{n+m}\Vert_2
	& \le \big(\Vert\overline{\calC}_n^{(1)}\Vert_2^2 + \Vert\overline{\calC}_{n,m}^{(2)}\Vert_2^2\big)^{1/2} + 4\Vert \mathcal{G}(\calR_n^{(1)},\calR_{n,m}^{(2)})\Vert_2 \\
	&\leq \big(\Vert\overline{\calC}_n^{(1)}\Vert_2^2 + \Vert\overline{\calC}_{n,m}^{(2)}\Vert_2^2\big)^{1/2} + c\mathcal{E}(n+m)
	\end{align*} for some $c>0$, where
	in the last inequality we used $$\mathcal{G}(\calR_n^{(1)},\calR_{n,m}^{(2)})\le \mathcal{G}(\calR_{n+m}^{(1)},\calR_{n+m,n+m}^{(2)}).$$ 
	Consequently,  
	\begin{align*}
	\Vert\overline{\calC}_{n + m}\Vert_2^2
	\le \Vert\overline{\calC}_n\Vert_2^2 + \Vert\overline{\calC}_m\Vert_2^2 + 2c \big(\Vert\overline{\calC}_n\Vert_2^2 + \Vert\overline{\calC}_m\Vert_2^2\big)^{1/2} \mathcal{E}(n+m) + c^2\mathcal{E}(n+m)^2. 
	\end{align*}
	By setting $g(n)=\Vert\overline{\calC}_{n}\Vert_2^2$ for $n\in\N$,  
	the above relation reads 
	\begin{equation}\label{to show_1}
g(n+m)\le g(n)+g(m)+2c \big(\Vert\overline{\calC}_n\Vert_2^2 + \Vert\overline{\calC}_m\Vert_2^2\big)^{1/2} \mathcal{E}(n+m) + c^2\mathcal{E}(n+m)^2.
		\end{equation}
		In the sequel we find an upper bound for the third term of the right hand side of inequality \eqref{to show_1}.  If we  prove that
\begin{equation}\label{eqc}
\Vert \overline{\calC}_n\Vert_2\lesssim\sqrt{n}\,\mathcal{E}(n),
\end{equation}	
then by defining $\varphi(n)\asymp\sqrt{n}\mathcal{E}(n)^2$ the assertion of the lemma will follow directly from \eqref{to show_1}, the fact that $d\ge6$ and de Brujin-Erd\H{o}s extension of Fekete's lemma (see \cite[Theorem 23]{Brujin}).

Let us prove \eqref{eqc}. For  $k \in\N$  we set
	\begin{equation}\label{eq:def_of_a_k}
	\gamma(k) = \sup \left\{ \Vert \overline{\calC}_i\Vert_2:\, 2^k \le i \le 2^{k + 1}\right\}.
	\end{equation}
	Further, for $k \ge 2$ we take $n\in\N$ such that $2^k \le n < 2^{k + 1}$, and we set $l = \floor{n/2}$ and $m = n - l$. 
	Analogously as above we have
	\begin{equation*}\label{eq:Cor_2.1.-existance_of_limit}
	\calC_l^{(1)} + \calC_{l,m}^{(2)}  -2G(\calR_l^{(1)},\calR_{l,m}^{(2)}) \le \calC_n \le \calC_l^{(1)} + \calC_{l,m}^{(2)},
	\end{equation*}
	and we arrive at
	\begin{align}\label{eq:inter}
	\Vert\overline{\calC}_n\Vert_2
	&  \le \big(\Vert\overline{\calC}_l^{(1)}\Vert_2^2 + \Vert\overline{\calC}_{l,m}^{(2)}\Vert_2^2\big)^{1/2} + c \mathcal{E}(n)
	\end{align} for some $c>0$.
	 Recall that $\calR_l^{(1)}$ and $\calR_{l,m}^{(2)}$ ($\calC_l^{(1)}$ and $\calC_{l,m}^{(2)}$) are independent and have the same law as $\calR_l$ and $\calR_m$ ($\calC_l$ and $\calC_m$), respectively.
	Hence, equation \eqref{eq:inter} implies
	\begin{equation*}
	\Vert\overline{\calC}_n\Vert_2 \le \sqrt{2} \gamma(k-1) + c \mathcal{E}(n).
	\end{equation*}
	Taking supremum over $2^k \le n \le 2^{k + 1}$ yields	\begin{equation*}\label{eq:rec_relation_for_a_k}
	\gamma(k) \le \sqrt{2} \gamma(k-1) + c \mathcal{E}(2^{k + 1}).
	\end{equation*}
	We next set $\delta(k) = \gamma(k) / \mathcal{E}(2^k)$.  Due to monotonicity of $\mathcal{E}(n)$ and the fact that $\mathcal{E}(2^{k+1})\le2 \mathcal{E}(2^k)$ for all $n\in\N$ (recall the definition of $\mathcal{E}(n)$) it holds that
	\begin{equation*}
	\delta(k) \le \sqrt{2} \delta(k-1) + 2c.
	\end{equation*}
	By iteration of this inequality we have $\delta(k) \lesssim 2^{k/2}$ which  implies $\gamma(k) \lesssim 2^{k/2} \mathcal{E}(2^k)$. Finally, using definition of $\gamma(k)$  we obtain
$$
\lVert\overline\calC_n\rVert_2^2 \le \gamma(k)^2 \lesssim  2^k \mathcal{E}(2^k)^2 \lesssim  n\, \mathcal{E}(n)^2,
	$$
which is \eqref{eqc}.
\end{proof}


Denote the limit of $\{\Var(\calC_n)/n\}_{n\ge1}$ by $\sigma_d^2$. Strict positivity of  $\sigma_d^2$ now follows  from (a straightforward adaptation of) \cite[Section 3.3]{Asselah_Zd} where $\{S_n\}_{n\ge0}$ is decomposed into two independent processes. The first process is the process counting double-backtracks at even times and it is given as a sequence of i.i.d.\ geometric random variables $\{\xi_i\}_{i \in2\N}$  with parameter $p=1 - 1/|\Gamma|^2$
 (we say that a random walk $\{\widetilde S_n\}_{n\ge0}$ on $\GG$ makes a double backtrack at even time $n\in2\N$ if $\widetilde S_{n - 1} = \widetilde S_{n - 3}$ and $\widetilde S_n = \widetilde S_{n - 2}$). The second process is a random walk $\{\hat S_n\}_{n\ge0}$ on $\GG$
with no double-backtracks at even times constructed as follows: set $\hat{S}_0 = 0$ and then choose $\hat{S}_1$ and $\hat{S}_2$ in the same fashion as one would choose them in a symmetric random walk on $\GG$. Suppose that we have constructed $\{\hat S_n\}_{n\ge0}$ for all times $k \le 2n$. We now let $(\hat{S}_{2n + 1}, \hat{S}_{2n + 2})$ be uniform in the set
\begin{equation*}
    \{(g_1, g_2) \in \GG \times \GG : \rho(g_1, g_2) = 1,\, \rho(\hat{S}_{2n}, g_1) = 1\, \textnormal{ and }\, (g_1, g_2) \neq (\hat{S}_{2n - 1}, \hat{S}_{2n}) \}.
\end{equation*} Further, setting $N_0=N_1=0$ and
\begin{equation*}
    N_k = \sum_{\substack{i = 2 \\ i \in2\N}}^k \xi_i
\end{equation*}
for $k\ge2$, we construct $\{\tilde S_n\}_{n\ge0}$ from $\{\hat S_n\}_{n\ge0}$ as follows. First we set $\tilde S_i = \hat{S}_i$ for  $i=0,1,2$, and for all $k \in\N$ we set $I_k = [2k + 2N_{2(k - 1)} + 1, 2k + 2N_{2k}]$. If $I_k \neq \emptyset$, then if $i \in I_k$ is odd, we set $\tilde S_i = \hat{S}_{2k - 1}$, while if $i$ is even, we set $\tilde S_i = \hat{S}_{2k}$. Afterwards, for the next two steps, we follow the path of $\{\hat S_n\}_{n\ge0}$, i.e.
\begin{equation*}
     \tilde S_{2k + 2N_{2k} + 1} = \hat{S}_{2k + 1} \qquad \textnormal{and}\qquad \tilde S_{2k + 2N_{2k} + 2} = \hat{S}_{2k + 2}.
\end{equation*}
It is now immediate that $\{\tilde S_n\}_{n\ge0}$ is a symmetric random walk on $\GG$, i.e.\ identical in law to $\{S_n\}_{n\ge0}$.
Strict positivity of $\sigma_d^2$ now follows by performing an analogous analysis as in \cite[Lemma 3.3]{Asselah_Zd}.

 We finally have:

\begin{proof}[Proof of Theorem \ref{CLT}]
The proof proceeds analogously as for \cite[Theorem 1.1]{Asselah_Zd}. We explain the main steps of the proof only. The proof is based on the following three results:
\begin{itemize}
\item [(a)] dyadic capacity decomposition formula (generalisation of formula \eqref{decomp}) derived in \cite[Corollary 2.1]{Asselah_Zd}:  Let $L,n\ge1$ be such that $2^L\le n$.  Then, \begin{equation}\label{eq:dyadic}\sum_{i=1}^{2^L}\mathrm{Cap}\,(\calR^{(i)}_{n/2^L})-2\sum_{l=1}^L\sum_{i=1}^{2^{l-1}}\mathrm{Er}_l^{(i)}
		\le \calC_n\le \sum_{i=1}^{2^L}\mathrm{Cap}\,(\calR^{(i)}_{n/2^L}),\end{equation} where $\{\calR^{(i)}_{n/2^L}\}_{i=1,\dots,2^L}$ are independent and  $\calR^{(i)}_{n/2^L}$ has the same law as $\calR_{\floor{n/2^L}}$ or $\calR_{\floor{n/2^L+1}}$, and for each $l=1,\dots,2^L$ the random variables $\{\mathrm{Er}_l^{(i)}\}_{i=1,\dots,2^{l-1}}$
		are independent and $\mathrm{Er}_l^{(i)}$ has the same law as $\mathcal{G}(\calR^{(i)}_{n/2^l},\bar\calR^{(i)}_{n/2^l})$ with  $\{\bar\calR_n\}_{n\ge0}$ being an independent copy of $\{\calR_n\}_{n\ge0}$.

			\medskip
			
			\item[(b)] fourth moment estimate of the capacity derived in \cite[Lemma 4.2]{Asselah_Zd}: If $d\ge6$, then
\begin{equation}\label{eq_4}	\E\big[  \overline{\calC}_n^4 \big] \lesssim n^2,
	\end{equation} where again $\overline\calC_n=\calC_n-\E[\calC_n].$
			
			\medskip
		
		\item[(c)]  Lindeberg-Feller central limit theorem (see \cite[Theorem 4.5]{Durrett}): For  $n\in\N$ let $\{X_{n,i}\}_{1\le i\le n}$ be a sequence of independent random variables. If
	\begin{itemize}
		\item [(i)] $\displaystyle\lim_{n\to\infty}\sum_{i=1}^n\Var(X_{n,i})=\sigma^2>0;$

		\item [(ii)] for every $\varepsilon>0$, $\displaystyle\lim_{n\to\infty}\sum_{i=1}^n\E\left[(X_{n,i}-\E[X_{n,i}])^2\bbjedan_{\{|X_{n,i}-\E[X_{n,i}|>\varepsilon\}}\right]=0,$
	\end{itemize}
	then $X_{n,1}+\cdots+ X_{n,n}\xrightarrow[n\to\infty]{\rm{(d)}}\calN(0,\sigma^2).$ 
\end{itemize}
We remark that both \cite[Corollary 2.1 and Lemma 4.2]{Asselah_Zd} are stated only for the symmetric simple random walk on $\ZZ^d$, but the proofs are valid for any finitely generated group with symmetric set of generators.

Denote $\calC^{(i)}_{n/2^L}= \mathrm{Cap}\,(\calR^{(i)}_{n/2^L})$ for $i=1,\dots,2^L$.
	By taking expectation in \eqref{eq:dyadic} and then subtracting those two relations we obtain	
	\begin{equation}\label{eq:relation_from_Cor_2.1-with_bars}
	\sum_{i = 1}^{2^L} \overline\calC^{(i)}_{n/2^L} - 2 \sum_{l = 1}^L \sum_{i = 1}^{2^{l - 1}} \mathrm{Er}_l^{(i)} \le \overline{\calC}_n \le \sum_{i = 1}^{2^L} \overline\calC^{(i)}_{n/2^L} + 2 \sum_{l = 1}^L \sum_{i = 1}^{2^{l - 1}} \E[\mathrm{Er}_l^{(i)}],
	\end{equation} where $\overline\calC^{(i)}_{n/2^L}=\calC^{(i)}_{n/2^L}-\E[\overline\calC^{(i)}_{n/2^L}]$ for $i=1,\dots,2^L$.
	For $n\in\N$, define
	\begin{equation*}
	\mathrm{Er}(n) = \sum_{i = 1}^{2^L} \overline\calC^{(i)}_{n/2^L} - \overline{\calC}_n.
	\end{equation*}
	Using \eqref{eq:relation_from_Cor_2.1-with_bars} and \eqref{eq:err}, we get that
	\begin{equation*}
	\E\big[ \aps{\mathrm{Er}(n)} \big] \le 4 \E\ugl{\sum_{l = 1}^L \sum_{i = 1}^{2^{l - 1}} \mathrm{Er}_l^{(i)}} \lesssim\sum_{l = 1}^L \sum_{i = 1}^{2^{l - 1}} \mathcal{E} \obl{\frac{n}{2^l}} \le \mathcal{E}(n) \sum_{l = 1}^L 2^{l-1} \le 2^{L} \log(n)
	\end{equation*} (recall the definition of the function $\mathcal{E}(n)$ from Proposition \ref{lm:bounds_on_GGG}).
Set now $L =L(n)=\lfloor\log_2\big(n^{1/4} \big)\rfloor$. Then, 
	$$
	\lim_{n\to \infty} \frac{\E\big[\aps{\mathrm{Er}(n)}\big] }{\sqrt{n} } =0.
$$	We are thus left to prove that
	\begin{equation*}
	\sum_{i = 1}^{2^L} \frac{\overline\calC^{(i)}_{n/2^L}}{\sqrt{n}}\xrightarrow[n\to\infty]{\text{(d)}}   \calN(0, \sigma_d^2).
	\end{equation*} 
	To establish this result we apply the Lindeberg-Feller central limit theorem. For $n\in\N$ and $1\le i\le 2^L$ put $X_{2^L,i}=\overline\calC^{(i)}_{n/2^L}/\sqrt{n}.$
	By Proposition \ref{prop3.x} we have
	\begin{equation*}
\lim_{n\to \infty }\sum_{i = 1}^{2^L}\Var(X_{2^L,i})=	\lim_{n\to \infty }\sum_{i = 1}^{2^L} \frac{1}{n} \Var(\overline\calC^{(i)}_{n/2^L}) = \sigma_d^2,
	\end{equation*}
	which means that the first Lindeberg-Feller condition is satisfied. 
	It remains to check that for any $\varepsilon >0$ it holds that
	\begin{equation*}\lim_{n\to\infty}\sum_{i=1}^{2^L}\E\left[X_{2^L,i}^2\bbjedan_{\{|X_{2^L,i}|>\varepsilon\}}\right]=
	\lim_{n \to \infty} \sum_{i = 1}^{2^L} \frac{1}{n} \E\left[ \left(\overline\calC^{(i)}_{n/2^L}\right)^2 \bbjedan_{\{ \aps{\overline\calC^{(i)}_{n/2^L}} > \varepsilon \sqrt{n} \}}\right] = 0.
	\end{equation*}
	Observe that by the Cauchy-Schwartz inequality we have 
	\begin{align*}
	\E\left[ \left(\overline\calC^{(i)}_{n/2^L}\right)^2 \bbjedan_{\{ \aps{\overline\calC^{(i)}_{n/2^L}} > \varepsilon \sqrt{n} \}}\right]
	&\le \left(\E\left[\left(\overline\calC^{(i)}_{n/2^L}\right)^4\right] \Prob(\aps{\overline\calC^{(i)}_{n/2^L}} > \varepsilon \sqrt{n})\right)^{1/2}.
	\end{align*}
Further, the Chebyshev inequality combined with \eqref{eq_4}, strict positivity of $\sigma_d$ and Proposition \ref{prop3.x}, imply
	\begin{align*}
\E\left[\left(\overline\calC^{(i)}_{n/2^L}\right)^4\right] \Prob(\aps{\overline\calC^{(i)}_{n/2^L}} > \varepsilon \sqrt{n})
	 \lesssim\obl{\frac{n}{2^L}}^2 \, \frac{\Var\left(\overline\calC^{(i)}_{n/2^L}\right)}{\varepsilon^2 n} 
	 \lesssim 2\sigma_d^2 \frac{n^2}{\varepsilon^2 2^{3L}} .
	\end{align*}
	Based on the choice  $L =\lfloor\log_2\big(n^{1/4} \big)\rfloor $,  we conclude 
	\begin{equation*}
	\sum_{i = 1}^{2^L} \frac{1}{n} \E\left[ \left(\overline\calC^{(i)}_{n/2^L}\right)^2 \bbjedan_{\{ \aps{\overline\calC^{(i)}_{n/2^L}} > \varepsilon \sqrt{n} \}}\right] \lesssim n^{-1/8},
	\end{equation*}
and this finishes the proof.

\end{proof}

\begin{remark}
    Having the CLT at our disposal, under the same assumptions we can also obtain a functional version of this result. It is straightforward to reproduce the arguments from  \cite[Theorem 1.1]{CSS-FCLT} (by checking the Aldous' tightness condition and showing  convergence of finite dimensional distributions in \eqref{eq:fclt}) and conclude the following
    \begin{equation}\label{eq:fclt}
        \VIT{\frac{\calC_{\floor{nt}} - \E[\calC_{\floor{nt}}]}{\sigma_d \sqrt{n}}}_{t \ge 0} \xrightarrow[n \to \infty]{(\textnormal{J}_1)}\, \{B_t\}_{t \ge 0}\,,
    \end{equation}
    where $\sigma_d$ is the constant from Theorem \ref{CLT}, $\xrightarrow[n \to \infty]{(\textnormal{J}_1)}$ stands for the weak convergence in  the Skorohod space $\calD([0, \infty), \bbR)$ endowed with the $\textnormal{J}_1$ topology, and $\{B_t\}_{t \ge 0}$ denotes a standard one-dimensional Brownian motion.
\end{remark}

\section*{Acknowledgement} 
\noindent
We would like to thank an anonymous referee for numerous valuable comments. This work has been supported by \textit{Deutscher Akademischer} \textit{Austauschdienst} (DAAD) and \textit{Ministry of Science and Education of the Republic of Croatia} (MSE) via project \textit{Random Time-Change and Jump Processes}.  
Financial support through the  \textit{Croatian Science Foundation} under project 4129 (for R.\ Mrazovi\'{c}), 
\textit{Alexander-von-Humboldt Foundation} under project No.\ HRV 1151902 HFST-E and \textit{Croatian Science Foundation} under project 8958 (for N.\ Sandri\'c), and \textit{Croatian Science Foundation} under project 4197 (for S.\ \v Sebek) is gratefully acknowledged.


\bibliographystyle{babamspl}
\bibliography{Capacity_on_groups}

\end{document}